\documentclass[11pt,reqno]{amsart}
\usepackage{amsmath}
\usepackage{amsthm,amscd,amssymb,amsfonts, amsbsy}
\usepackage{latexsym}
\usepackage{mathrsfs}
\usepackage{pxfonts}
\usepackage{color}
\usepackage{enumerate}
\UseRawInputEncoding

\def\XXint#1#2#3{{\setbox0=\hbox{$#1{#2#3}{\int}$}
     \vcenter{\hbox{$#2#3$}}\kern-.5\wd0}}

\def \W{\textbf{W}}

\date{\today}

\numberwithin{equation}{section}

\theoremstyle{plain}
\newtheorem{theorem}{Theorem}[section]
\newtheorem{lemma}{Lemma}[section]
\newtheorem{corollary}{Corollary}[section]
\newtheorem{proposition}{Proposition}[section]

\theoremstyle{definition}

\newtheorem*{acknowledgment}{Acknowledgment}

\theoremstyle{remark}
\newtheorem{remark}{Remark}[section]

\newenvironment{pfthm2}{{\par\noindent
            \textbf{Proof of Theorem \ref{thm1}}\quad}}{}

\newenvironment{pfthm3}{{\par\noindent
            \textbf{Proof of Theorem \ref{Theo:Lorentz}}\quad}}{}

 \newcommand{\ba}{\begin{aligned}}
 \newcommand{\ea}{\end{aligned}}

\newcommand{\bR}{\mathbb R}

\renewcommand{\epsilon}{\varepsilon}

\begin{document}
\title[3D inhomogeneous incompressible Navier-Stokes equations]{Regularity criteria and Liouville Theorem for 3D inhomogeneous Navier-Stokes flows
with vacuum}

\author[J.-M. Kim]{Jae-Myoung Kim}
\address[Jae-Myoung Kim]{\newline
Department of Mathematics Education, Andong National University
\newline Andong 36729, Republic of Korea}\email{jmkim02@anu.ac.kr}

\thanks{}


\maketitle

\begin{abstract}
In this paper, we investigate the 3D inhomogeneous Navier-Stokes
flows with vacuum, and obtain regularity criteria and Liouville type
theorems in the Lorentz space if a smooth solution $(\rho,
\mathbf{u})$ satisfies suitable conditions.
\end{abstract}

\section{Introduction}
We consider the existence for solutions $(\rho,u, \pi):
Q_T\rightarrow \bR\times \bR^3\times\bR$ to the three-dimensional
incompressible magnetohydrodynamic (MHD) equations
\begin{equation}\label{IHMHD}
\left\{
\begin{array}{ll}
\partial_t\rho + u\cdot \nabla \rho = 0, \\
\rho\displaystyle u_t -\Delta u+\rho(u \cdot
\nabla) u   +\nabla \pi= 0\\
\vspace{-3mm}\\
\displaystyle \text{div} \ u =0,\\
\end{array}\right.
\,\,\, \mbox{ in } \,\,Q_T:=\bR^3\times [0,\, T).
\end{equation}
Here $\rho$ is the density function, $u$ is the flow velocity and
$\displaystyle\pi$ is the pressure. We consider the initial value
problem of \eqref{IHMHD}, which requires initial
\begin{equation}\label{MHD-30}
\rho(x,0)=\rho_0(x),\quad u(x,0)=u_0(x),  \qquad x\in\bR^3.
\end{equation}

Kazhikov \cite{Kazhikov74} proved that the inhomogeneous
Navier-Stokes equations \eqref{IHMHD}--\eqref{MHD-30} have at least
one global weak solution in the energy space for smooth data with no
vacuum. After that, Ladyzhenskaya and Solonnikov \cite{LS78} first
established the unique solvability for the system with smooth
initial data that has no vacuum, global well-posedness in two
dimensions and local well-posedness in three dimensions. Moreover,
if he initial data is small enough, then global well-posedness holds
true. Simon \cite{Simon90} constructed global weak solutions to the
system with finite energy with the statues containing vacuum. (see
e.g. Danchin \cite{Danchin03}, \cite{Danchin04} for the almost
critical Sobolev spaces and Abidi, Gui and Zhang \cite{AGZ11} for
axi-symmetric initial data and Mucha, Xue and Zheng\cite{MXZ19} and
references therein).

%

Regarding the regularity criteria for system
\eqref{IHMHD}--\eqref{MHD-30}, Kim \cite{Kim06} established the
following Serrin type condition
\[
u \in L^t(0, T;L^{s,\infty}(\bR^3)), \frac{3}{s}+\frac{2}{t} = 1, s
\in (3,\infty],
\]
in the framework of Lorentz space (see e.g. Ye and Zhang
\cite{YZ15}, Sun and Qian \cite{SQ21} for Besov space)

Recall first the definition of Lorentz spaces. Let $m(\varphi,t) $
be the Lebesgue measure of the set $\{x\in \bR^3:|\varphi(x)|> t\}$,
\emph{i.e.}
$$
m(\varphi,t):=m\{x\in \bR^3:|\varphi(x)|> t\}.
$$ We denote by the
Lorentz space $L^{p,q}(\bR^3)$ with $1\leq p$, $q\leq \infty $ with
the norm \cite{Tr}
\begin{eqnarray*}\label{poiseuille}
\|\varphi\|_{L^{p,q}(\bR^3)}=\left\{
\begin{aligned}
&\Big(\int_0^{\infty}t^q(m(\varphi,t))^{q/p}\ \frac{dt}{t}
\Big)^{1/q}<\infty,\quad \mbox{for }\ 1\leq q<\infty,\\
&\sup_{t\geq 0}\{t(m(\varphi,t))^{\frac1p}\}<\infty,\,\hspace{1.5cm}
\ \ \quad \mbox{for }\  q=\infty
\end{aligned}\right.
\end{eqnarray*}
Note that $L^{r} (\bR^3) = L^{r,r}(\bR^3) \subset L^{r,q}(\bR^3)$ for $1<r < q \leq +\infty$. \\

In this direction, first result is stated as
\begin{theorem}\label{thm1} Let $(\rho, u)$ be the unique local strong solution in time
interval $[0, T)$ to the system \eqref{IHMHD}--\eqref{MHD-30} with
initial data (1.2). Then there exists a positive constant
$\varepsilon$ such that $(\rho, u)$ is a regular solution on $(0,T]$
provided that one of the following two conditions holds
\begin{enumerate}[(A)]
 \item
   $  u \in L^{q,\infty}(0,T; L ^{p,\infty}(\mathbb{R}^{3}))$ ~and~ $$\|u\|_{L^{q,\infty}(0,T; L ^{p,\infty}(\mathbb{R}^{3}))}
     \leq\varepsilon, ~ \text{with} ~~\frac{3}{p}+\frac{2}{q}=1 , ~3<p<\infty;  $$
 \item   $$\|\nabla u\|_{L^{q}(0,T; L ^{p,\infty}(\mathbb{R}^{3}))}
     <\infty, ~ \text{with} ~~\frac{3}{p}+\frac{2}{q}=2 , ~\frac{3}{2}<p<\infty;  $$
 \end{enumerate}
\end{theorem}

\begin{corollary}\label{cor1} Let $(\rho, u)$ be the unique local strong solution in time
interval $[0, T)$ to the system \eqref{IHMHD}--\eqref{MHD-30} with
initial data (1.2). Suppose that $u$ satisfies the following
conditions holds
$$\|u\|_{L^{q}(0,T; L ^{p,\infty}(\mathbb{R}^{3}))}
     <\infty, ~ \text{with} ~~\frac{3}{p}+\frac{2}{q}=1 , ~3<p<\infty,  $$
then the solution pair $(\rho, u)$ can be extended beyond time
$T>0$.
\end{corollary}

\begin{remark}
Theorem \ref{thm1} and Corollary \ref{cor1} seem to hold for bounded
domains with smooth boundary with Dirichlet or slip type boundary
conditions.
\end{remark}

%
%

On the other hands, regarding Liouville theorems for the fluid flows
contains the density function, in particular, compressible
Naiver-Stokes equations, Chae \cite{Chae2012} showed if the smooth
solution $(\rho, u)$ satisfies $
\|\rho\|_{L^\infty(\mathbb{R}^3)}+\|\nabla
u\|_{L^2(\mathbb{R}^3)}+\|u\|_{L^\frac{3}{2}(\mathbb{R}^3)}<\infty.
$ then $u\equiv 0$ and $\rho=constant$ (see Li and Yu \cite{L2014}).
In the Lorentz framework, very recently, Li and Niu \cite{LN21}
proved if $(\rho, u)$ is a smooth solution with $\rho\in L^\infty
(\mathbb{R}^3)$, $\nabla u\in L^2(\mathbb{R}^3)$ and $u\in L^{p,
q}(\mathbb{R}^3)$ for $3< p <\frac{9}{2}$, $3\leq q\leq \infty$ or
$p=q=3$. Then $u\equiv 0$ and $\rho=constant$ in $\mathbb{R}^{3}$.
For the equations \eqref{IHMHD}--\eqref{MHD-30}, to the best of my
knowledge, There is no known result so far, even in $L^p$-space. In
this direction, parallel to the result of Jarr\'{i}n
\cite{Jarrin20}, second result is stated as

\begin{theorem}\label{Theo:Lorentz} Let $(\rho, u)$ be a weak
solution to the stationary equation corresponding to \eqref{IHMHD}.
If $\rho \in L^\infty(\bR^3)$, $\nabla u \in L^{2}(\bR^3)$ and $u\in
L^{\frac{9}{2},q}(\bR^3)$ with $\frac{9}{2}\leq q<\infty$, then we
have $u\equiv0$ and $\rho\equiv0$ in $\bR^3$.
\end{theorem}

Comparing to \cite{Jarrin20}, \ when $\frac{9}{2}<r$, it is not
shown Liouville Theorem because it is difficult to control for
$\nabla u$ due to the low regularity for density function $\rho$. in
this case, it is not yet known whether the result is valid for the
same reason.

\section{Preliminaries}      \label{sec_mr}
In this section we introduce some the definition of functional space
and its properties. For $1\le q\le \infty$, we denote the usual
Sobolev spaces by $W^{k,q}(\bR^3) = \{ u \in L^{q}( \bR^3 )\,:\, D^{
p }u \in L^{q}( \bR^3 ), 0 \leq | p | \leq k \}$. When $q=2$, we
write $W^{k,q}(\bR^3)$ as $H^{k}(\bR^3)$. We review some useful
estimates for the proof of Theorems.

We need the H\"{o}lder inequality in Lorentz spaces (see \cite{ON}).
\begin{lemma}\label{oneil}
Assume $1\leq  p_1$, $p_2\leq \infty$, $1\leq  q_1$, $q_2\leq \infty
$ and $u\in L^{p_1,q_1}(\Omega)$,  $v\in L^{p_2,q_2}(\Omega)$.
 Then $uv\in L^{p_3,q_3}(\Omega)$ with
$ \frac1{p_3}=\frac1{p_1}+\frac1{p_2}$ and $\frac1{q_3}\leq
\frac1{q_1}+\frac1{q_2} $, and  the  inequality
\begin{equation}\label{2.3}
\|uv\|_{L^{p_3,q_3}(\Omega)}\leq C
\|u\|_{L^{p_1,q_1}(\Omega)}\|v\|_{L^{p_2,q_2}(\Omega)}
\end{equation}
is valid.
\end{lemma}

Recall the following useful Gronwall lemma required in our proof,
which is first shown by \cite{BPR} (see e.g. \cite{PY},
\cite{LRM16}).
  \begin{lemma}\label{2.1}
 Let $\phi$ be a measurable positive function defined on the interval $[0,T]$. Suppose that there  exists  $\kappa_{0}>0$ such that for all $0<\kappa<\kappa_{0}$ and a.e. $t\in[0,T]$, $\phi$ satisfies the inequality
 $$\frac{d}{dt}\phi\leq \mu\lambda^{1-\kappa}\phi^{1+2\kappa},$$
 where   $0 <\lambda \in L^{1,\infty}(0,T)$ and $\mu> 0$  with
 $\mu\|\lambda\|_{L^{1,\infty}(0,T)}<\frac{1}{2}$.
 Then $\phi$ is bounded on $[0,T]$.
 \end{lemma}

\section{Proof of Theorems}

\begin{pfthm2}
By the $L^2$-energy estimate, we know
\[
\frac{d}{dt}\int_{\bR^3} \rho\|u(\cdot, s)\|_{L^2}+\|b(\cdot,
s)\|_{L^2}\,ds\leq 0,
\]
which implies
\begin{equation}\label{energy-l2}
u,b \in L^\infty(0,T; L^2(\bR^3)).
\end{equation}
(\textbf{Proof of (A)}):  Taking  $u_t$ to the first equation in
\eqref{IHMHD}, testing by $u_t$, and integrating over whole space,
we easily derive
\begin{equation}\label{thm1_01}
\frac{1}{2}\frac{d}{dt}\|\nabla u\|^2_{L^2} + \int_{\bR^3} \rho
|u_t|^2\,dx \leq \int_{\bR^3} |\rho^{1/2} u \cdot \nabla u\cdot
\rho^{1/2} u_t|\, dx:=\mathcal{I}.
\end{equation} Applying
H\"{o}lder��s and Lemma \ref{oneil}, we obtain
\[
\frac{1}{2}\frac{d}{dt}\|\nabla u\|^2_{L^2} + \int_{\bR^3} \rho
|u_t|^2\,dx \leq  C\|u\|_{L^{p,\infty}(\bR^3)}\|\nabla
u\|_{L^{\frac{2p}{p-2},2}(\bR^3)}\|\sqrt{\rho}u_t\|_{L^{2,2}(\bR^3)}
\]
\[
\leq  C\|u\|^2_{L^{p,\infty}(\bR^3)}\|\nabla
u\|^2_{L^{\frac{2p}{p-2},2}(\bR^3)}+\epsilon\|\sqrt{\rho}u_t\|^2_{L^{2}(\bR^3)}
\]
\[
\leq C\| u\|_{L^{p,\infty}(\bR^3)}\| \nabla
u\|^{2(1-\frac{3}{p})}_{L^{2}(\bR^3)}\|\nabla^2
u\|^{\frac{6}{p}}_{L^{2}(\bR^3)}+\epsilon\|\sqrt{\rho}u_t\|^2_{L^{2}(\bR^3)}
\]
\begin{equation}\label{est-1000}
\leq C\|u\|^{\frac{2p}{p-3}}_{L^{p,\infty}(\bR^3)}\|\nabla
u\|^{2}_{L^{2}(\bR^3)}+\epsilon(\|\nabla^2
u\|^2_{L^{2}(\bR^3)}+\|\sqrt{\rho}u_t\|^2_{L^{2}(\bR^3)})
\end{equation}
Here, $\epsilon>0$ is later determined. By using the classical
regularity theory to the second equation in \eqref{IHMHD} and Lemma
\ref{oneil}, it immediately implies that
\[
\|\nabla^2 u\|^2_{L^2(\bR^3)}\leq C
(\|\sqrt{\rho}u_t\|^2_{L^2(\bR^3)} + \|\rho u \cdot \nabla
u\|^2_{L^2(\bR^3)} )
\]
\[
\leq C(\|\sqrt{\rho}u_t\|^2_{L^2(\bR^3)}+
\|\rho\|^2_{L^\infty}\|u\|^2_{L^{p,\infty}(\bR^3)}\|\nabla
u\|^2_{L^{\frac{2p}{p-2},2}(\bR^3)})
\]
\begin{equation}\label{est-2000}
\leq C(\|\sqrt{\rho}u_t\|^2_{L^2(\bR^3)}+
C\|u\|^{\frac{2p}{p-3}}_{L^{p,\infty}(\bR^3)}\|\nabla
u\|^{2}_{L^{2}(\bR^3)})+\frac{1}{16}\|\nabla^2 u\|^2_{L^{2}}.
\end{equation}
Inserting \eqref{est-2000} into \eqref{est-1000} and choosing
$\epsilon>0$ such that $C\epsilon<\frac{1}{8}$, we see that
\[
\frac{1}{2}\frac{d}{dt}\|\nabla u\|^2_{L^2(\bR^3)} + \int_{\bR^3}
\rho |u_t|^2\,dx \leq
C\|u\|^{\frac{2p}{p-3}}_{L^{p,\infty}(\bR^3)}\|\nabla
u\|^{2}_{L^{2}(\bR^3)}.
\]
For $\kappa > 0$, let $q_\kappa = q + \kappa(4 - s)$ and $p_\kappa$
such that $(p_\epsilon, q_\kappa)$ satisfies
$\frac{3}{p_\kappa}+\frac{2}{q_\kappa}=1$. Through applying
interpolation, we have
\[
\|u\|_{L^{p_\kappa,\infty}(\bR^3)}^{q_\kappa}\leq
\|u\|_{L^{p,\infty}(\bR^3)}^{q(1-\kappa)}\|u\|^4_{L^{6,\infty}(\bR^3)}
\leq C \|u\|_{L^{p,\infty}(\bR^3)}^{q(1-\kappa)}\|\nabla
u\|_{L^2(\bR^3)}^4.
\]
Then, it yields that
$$\ba
\frac{d}{dt}\|\nabla u\|^2_{L^2(\bR^3)}
  \leq C \|u\|^{q(1-\kappa)}_{L^{p ,\infty}(\mathbb{R}^{3})} \|\nabla u\|_{L^2(\bR^3)}^{2(1+2\kappa)}.
\ea$$ Now, to invoke  Lemma \ref{2.1} with $\phi=\|\nabla
u\|^2_{L^2(\bR^3)}$  (see \cite[pp.6]{LRM16} for a detailed proof)
and therefore, the proof is complete.

(\textbf{Proof of (B)}): This proof is almost same to that Part
$(A)$. Indeed, from the estimate \eqref{thm1_01}, we know
\[
\frac{1}{2}\frac{d}{dt}\|\nabla u\|^2_{L^2} + \int_{\bR^3} \rho
|u_t|^2\,dx \leq \int_{\bR^3} |\rho^{1/2} u \cdot \nabla u\cdot
\rho^{1/2} u_t|\, dx:=\mathcal{I}.
\]
Applying H\"{o}lder��s and Lemma \ref{oneil}, we obtain
\[
\frac{1}{2}\frac{d}{dt}\|\nabla u\|^2_{L^2} + \int_{\bR^3} \rho
|u_t|^2\,dx \leq  C\|\nabla u\|_{L^{p,\infty}(\bR^3)}\|
u\|_{L^{\frac{2p}{p-2},2}(\bR^3)}\|\sqrt{\rho}u_t\|_{L^{2,2}(\bR^3)}
\]
\[
\leq  C\|\nabla
u\|^2_{L^{p,\infty}(\bR^3)}\|u\|^2_{L^{\frac{2p}{p-2},2}(\bR^3)}+\epsilon\|\sqrt{\rho}u_t\|^2_{L^{2}(\bR^3)}
\]
\[
\leq C\|\nabla u\|^2_{L^{p,\infty}(\bR^3)}\|\nabla
u\|^{2(2-\frac{3}{p})}_{L^{2}(\bR^3)}\|\nabla^2
u\|^{2(\frac{3}{p}-1)}_{L^{2}(\bR^3)}+\epsilon\|\sqrt{\rho}u_t\|^2_{L^{2}(\bR^3)}
\]
\begin{equation}\label{est-3000}
\leq C\|u\|^{\frac{2p}{2p-3}}_{L^{p,\infty}(\bR^3)}\|\nabla
u\|^{2}_{L^{2}(\bR^3)}+\epsilon(\|\nabla^2
u\|^2_{L^{2}(\bR^3)}+\|\sqrt{\rho}u_t\|^2_{L^{2}(\bR^3)})
\end{equation} By using the classical regularity theory, it
immediately implies that
\[
\|\nabla^2 u\|^2_{L^2(\bR^3)}\leq C
(\|\sqrt{\rho}u_t\|^2_{L^2(\bR^3)} + \|\rho u \cdot \nabla
u\|^2_{L^2(\bR^3)} )
\]
\begin{equation}\label{est-4000}
\leq C(\|\sqrt{\rho}u_t\|^2_{L^2(\bR^3)}+
C\|u\|^{\frac{2p}{2p-3}}_{L^{p,\infty}(\bR^3)}\|\nabla
u\|^{2}_{L^{2}(\bR^3)})+\frac{1}{16}\|\nabla^2 u\|^2_{L^{2}}.
\end{equation} Inserting \eqref{est-4000} into \eqref{est-3000}, we see that
\[
\frac{1}{2}\frac{d}{dt}\|\nabla u\|^2_{L^2(\bR^3)} + \int_{\bR^3}
\rho |u_t|^2\,dx \leq
C\|u\|^{\frac{2p}{2p-3}}_{L^{p,\infty}(\bR^3)}\|\nabla
u\|^{2}_{L^{2}(\bR^3)}.
\]
By Gronwall's inequality, we obtain the desired result, and thus the
proof is finally complete.
\end{pfthm2}

\bigskip

Next, to obtain the result of Liouville theorem, we start by
introducing the test functions $\varphi_R$ and $\omega_R$ as
follows:  for a fixed $R>1$, we define first the function
$\varphi_R\in \mathcal{C}^{\infty}_{0}(\bR^3)$ by $0\leq
\varphi_R\leq 1$ such that  for $\vert x \vert \leq \frac{R}{2}$ we
have $\varphi_R(x)=1$, for $\vert x \vert \geq R$ we have
$\varphi_R(x)=0$, and
\begin{equation}\label{control-norme-test}
\Vert \nabla \varphi_R \Vert_{L^{\infty}} \leq \frac{c}{R}.
\end{equation}

From Lemma $III. 3.1$ in \cite{Galdi11}, we recall that

\begin{lemma} Let $\varphi_R(x)$ be defined as above, then there
exist a vector-valued function $\omega_R$ and a constant $C(p)$ such
that
\begin{equation}\label{eq_W_R}
div(\W_R)=\nabla\varphi_R\cdot u, \quad \text{over}\,\, B_R, \quad
\text{and}\quad \W_R=0 \,\, \text{over}\,\,  \partial B_R \cup
\partial B_{\frac{R}{2}},
\end{equation}
 where $\partial B_r:=\{x \in \bR^3:|x| =r\}$ and $\|\omega_R\|_{W^{1,p}(B_r)}$ with
$\|\nabla  \omega_R\|_{L^p}\leq C(p)\|\nabla \varphi_R\cdot
u\|_{L^p}$ for $1 <p <+\infty$.
\end{lemma}

For second result, it is studied the following Caccioppoli type
estimate:
\begin{proposition}\label{Prop-Base} Let $\mathcal{C}(R/ 2, R) = \{ x
\in \bR^3: R/2 < \vert x \vert < R \}$. If the solution $(u,b)$
verifies $u \in L^{p}_{loc}(\bR^3)$ and $\nabla u \in
L^{\frac{p}{2}}_{loc}(\bR^3)$   with  $3 \leq p< +\infty$, then for
all $R>1$ we have
    \begin{equation}\label{Ineq-base}
    \begin{split}
      &  \int_{B_{R/2}} \vert \nabla u \vert^2
 dx    \lesssim \left( \int_{\mathcal{C}(R/2,R)}\vert \nabla u \vert^{\frac{p}{2}} dx+\int_{\mathcal{C}(R/2,R)} \vert
 u \vert^{p}dx \right)^{\frac{2}{p}} R^{2-\frac{9}{p}}
\left( \int_{\mathcal{C}(R/2,R)} \vert u \vert^{p} dx
\right)^{\frac{1}{p}}
      \end{split}
    \end{equation}
\end{proposition}
\begin{proof}
A proof is almost same to that \cite{Jarrin20}, we give a sketch of
proof for the convenience of the readers. We start by introducing
the test functions $\varphi_R$ and $W_R$ as follows: we have $W_R\in
W^{1,q}(B_R)$ with $supp\,(W_R)\subset \mathcal{C}(R/2, R)$ and
\begin{equation}\label{contril-Lq-test}
\Vert \nabla W_R\Vert_{L^q(\mathcal{C}(R/2, R))}\leq c \Vert
\nabla\varphi_R\cdot u \Vert_{L^q(\mathcal{C}(R/2, R))}.
\end{equation}

Once we have defined the functions $\varphi_R$ and $W_R$ above, we
consider now the function $\varphi_R u-W_R$ and we write
\begin{equation*}\label{eq01}
 \int_{B_R} \left( -\Delta u +(\rho u \cdot \nabla)u+\nabla \pi\right)\cdot \left( \varphi_R u-W_R \right)dx=0.
\end{equation*}


First of all, since $W_R$ is a solution of problem (\ref{eq_W_R})
and $u$ is divergence free vector, note that
$$ \int_{B_R} \nabla \pi \cdot \left( \varphi_R u-W_R \right)dx=0.$$

Thus by this inequality and the identity above we can write the
following estimate:
\begin{eqnarray} \label{eq03} \nonumber
 \int_{B_{R/2}} \vert \nabla u \vert^{2}\,dx & = &
-\sum_{i,j=1}^{3} \int_{B_R} \partial_j u_i (\partial_j \varphi_R)
u_i dx + \sum_{i,j=1}^{3} \int_{B_R} (\partial_j u_i )
\partial_j(W_R)_i dx  \\ \nonumber
 & &
+ \int_{B_R}(\rho u \cdot \nabla)u\cdot \left( \varphi_R u-W_R
\right)dx:=\sum_{i=1}^3\mathcal{K}_i.\\ \nonumber
 \end{eqnarray}
The term $ \mathcal{K}_2$ in (\ref{eq03}) is estimated by
$$ \mathcal{K}_2= \sum_{i,j=1}^{3} \int_{B_R} (\partial_j u_i ) \partial_j(W_R)_i dx =\sum_{i,j=1}^{3} \int_{\mathcal{C}(R/2,R)} (\partial_j u_i ) \partial_j(W_R)_i dx.$$
Applying H\"older's inequality and \eqref{contril-Lq-test}, we have
\begin{eqnarray*}
\mathcal{K}_2  &\lesssim&  \left( \int_{\mathcal{C}(R/2,R)} \vert
\nabla
 u \vert^{\frac{p}{2}} dx \right)^{\frac{2}{p}} \left(
\int_{\mathcal{C}(R/2,R)} \vert \nabla  \W_R
\vert^{q}dx \right)^{\frac{1}{q}}\\
&\lesssim&   \left( \int_{\mathcal{C}(R/2,R)} \vert \nabla u
\vert^{\frac{p}{2}} dx \right)^{\frac{2}{p}} \left(
\int_{\mathcal{C}(R/2,R)} \vert \nabla \varphi_R \cdot u
\vert^{q} dx \right)^{\frac{1}{q}}\\
&\lesssim&   \left( \int_{\mathcal{C}(R/2,R)} \vert \nabla u
\vert^{\frac{p}{2}} dx \right)^{\frac{2}{p}} R^{2-\frac{9}{p}}
\left( \int_{\mathcal{C}(R/2,R)} \vert u \vert^{p} dx
\right)^{\frac{1}{p}}
\end{eqnarray*}
Now We study each term in sequence. Due to the property for the
function $\partial_i \varphi_R$, we know that if $\vert x \vert
> R$ then we have $supp\,(\nabla \varphi_R) \subset
\mathcal{C}(R/2,R)$, and thus it is rewritten as
$$ \mathcal{K}_1= -\sum_{i,j=1}^{3} \int_{\mathcal{C}(R/2,R)} \partial_j u_i (\partial_j \varphi_R) u_i dx.$$
Applying the Hold\"er inequalities, we have
\begin{eqnarray}\label{eq06} \nonumber
\mathcal{K}_1 &\lesssim &  \left( \int_{\mathcal{C}(R/2,R)}\vert
\nabla u \vert^{\frac{p}{2}} dx \right)^{\frac{2}{p}} \frac{1}{R}
\left( \int_{\mathcal{C}(R/2,R)} \vert u \vert^{q}
 dx\right)^{\frac{1}{q}}\\ \nonumber
 &\lesssim &  \left( \int_{\mathcal{C}(R/2,R)}\vert \nabla u
\vert^{\frac{p}{2}} dx \right)^{\frac{2}{p}} R^{2-\frac{9}{p}}
\left( \int_{\mathcal{C}(R/2,R)} \vert u \vert^{p} dx
\right)^{\frac{1}{p}}.
\end{eqnarray} For $\mathcal{K}_3$, by the integrating by parts, we
know
\begin{eqnarray*}\label{eq09} \nonumber
\mathcal{K}_3& =& -\sum_{i,j=1}^{3} \int_{B_R} (\rho u_i u_j)
\Big((\partial_j \varphi_R) u_i + \varphi_R (\partial_j u_i)
-\partial_j(W_R)_i\Big) dx
\\ \nonumber
&=& - \sum_{i,j=1}^{3}  \int_{B_R} (\rho u_i u_j) \Big((\partial_j
\varphi_R) u_i dx -  (u_i u_j) \varphi_R (\partial_j u_i)
+  (u_i u_j) \partial_j(W_R)_i\Big) dx\\
&=& J_{1}+ J_{2}+J_{3}.
\end{eqnarray*}
Again, we estimate each term separately. In term $\mathcal{J}_{1}$,
in the previous way, it is easily checked that
\begin{eqnarray}\label{eq10}\nonumber
\mathcal{J}_{1} 
 &\lesssim & \left( \int_{\mathcal{C}(R/2,R)} \vert u\vert^{p}dx \right)^{\frac{2}{p}}  R^{2-\frac{9}{p}} \left(  \int_{\mathcal{C}(R/2,R)} \vert u \vert^{p} dx
 \right)^{\frac{1}{p}},
 \end{eqnarray}
 where we use $\|\rho\|_{L^\infty}<0$.
Next, to estimate the term $\mathcal{J}_{2}$, by the integration by
parts and divergence free condition for $u$, we know
\begin{equation*}
\mathcal{J}_{2}  = - \frac{1}{2} \sum_{i,j=1}^{3} \int_{B_R}  u_j
\varphi_R
\partial_j (u^{2}_{i}) dx= \frac{1}{2} \sum_{i,j=1}^{3}
\int_{\mathcal{C}(R/2,R)} u^{2}_{i} (\partial_j \varphi_R) u_j dx.
\end{equation*}
In the same way as $\mathcal{J}_{1}$, it follows that
\begin{equation*}\label{eq11}
\mathcal{J}_{2} \lesssim   \left( \int_{\mathcal{C}(R/2,R)} \vert
u\vert^{p}dx \right)^{\frac{2}{p}}  R^{2-\frac{9}{p}} \left(
\int_{\mathcal{C}(R/2,R)} \vert u \vert^{p} dx
\right)^{\frac{1}{p}}.
\end{equation*}
Similarly, we have
\begin{equation*}\label{eq20}
\mathcal{J}_{3} \lesssim   \left( \int_{\mathcal{C}(R/2,R)} \vert
u\vert^{p}dx \right)^{\frac{2}{p}}  R^{2-\frac{9}{p}} \left(
\int_{\mathcal{C}(R/2,R)} \vert u \vert^{p} dx
\right)^{\frac{1}{p}}.
\end{equation*}
Summing up $\mathcal{J}_{1}$--$\mathcal{J}_{3}$, we have
\[
\mathcal{K}_{3} \lesssim   \left( \int_{\mathcal{C}(R/2,R)} \vert
u\vert^{p}dx \right)^{\frac{2}{p}}  R^{2-\frac{9}{p}} \left(
\int_{\mathcal{C}(R/2,R)} \vert u \vert^{p} dx
\right)^{\frac{1}{p}}.
\]
Again, collecting up $\mathcal{K}_{1}$--$\mathcal{K}_{3}$, we
finally obtain the desired result.

\end{proof}

\begin{pfthm3}
For $1<p<r \leq q < +\infty$, we have
\begin{equation}\label{estim-Lorentz}
int_{B_R} \vert u \vert^{p} dx \leq c \, \, R^{3(1-\frac{p}{r})}
\Vert u \Vert^{p}_{L^{r,\infty}} \leq c   \, R^{3(1-\frac{p}{r})}
\Vert u \Vert^{p}_{L^{r,q}},\quad R>1,
\end{equation}
see \cite[Proposition 1.1.10]{Chamorro18}. From \eqref{Ineq-base},
we know that
 Through Lemma \ref{Prop-Base},  we write for all $R>1$
\begin{equation} \label{eq13}
\int_{B_{R/2}} \vert \nabla u \vert^2 dx\lesssim (G_{u}+S_{u})P_u,
\end{equation}
where
\[
G_{u}:=c \left( R^{\frac{6}{r}} \left( \frac{1}{R^3}
\int_{\mathcal{C}(R/2,R)} \vert \nabla u\vert^{\frac{p}{2}}dx
\right)^{\frac{2}{p}},\quad S_{u}:=R^{\frac{6}{r}}
\left(\frac{1}{R^3} \int_{\mathcal{C}(R/2,R)} \vert u\vert^{p}dx
\right)^{\frac{2}{p}} \right),
\]
\[
P_u:=R^{2-\frac{9}{r}}  \left( R^{\frac{3}{r}} \left( \frac{1}{R^3}
\int_{\mathcal{C}(R/2,R)} \vert u \vert^{p} dx
\right)^{\frac{1}{p}}\right).
\]
For  this  we introduce  the cut-off function  $\theta_R \in
\mathcal{C}^{\infty}_{0}(\bR^3)$ such that  $\theta_R =1$ on  $
\mathcal{C}(R/2,R)$, $supp\, (\theta_R) \subset \mathcal{C}(R/4,
2R)$ and $\Vert \nabla \theta_R \Vert_{L^{\infty}} \leq
\frac{c}{R}$. In the same arguments in \cite{Jarrin20}, \eqref{eq13}
yields
\begin{equation*}\label{eq18}
\int_{B_{\frac{B}{2}}} \vert \nabla u \vert^2 dx \leq c \left( \Vert
\theta_R (\nabla u)\Vert_{L^{\frac{r}{2},\frac{q}{2}}} + \Vert
\theta_R u \Vert^{2}_{L^{r,q}}\right) \, R^{2-\frac{9}{r}} \Vert u
\Vert_{L^{r,q}(\mathcal{C}(R/4,R))}.
\end{equation*}
Taking $p=4$ and $r=9/2$, it implies $u\equiv0$ and $\rho \equiv 0$
in $\bR^3$ as $R\rightarrow \infty $. The proof of Theorem
\ref{Theo:Lorentz} is complete.
\end{pfthm3}

\begin{acknowledgment}

Jae-Myoung Kim was supported by a Research Grant of Andong National
University.

\end{acknowledgment}



\begin{thebibliography}{m}

%
%
%
%



\bibitem{AGZ11} H. Abidi, G. Gui, P. Zhang, On the decay and stability to global
solutions of the 3-D inhomogeneous Navier.Stokes equations, Comm.
Pure Appl. Math. 64 (2011) 832--881.

 \bibitem{BPR}
S. Bosia, V. Pata and J. Robinson, A Weak-$L^p$ Prodi-Serrin Type
Regularity Criterion for the Navier-Stokes Equations. J. Math. Fluid
Mech.,  16 (2014), 721--725.

\bibitem{Chae2012} D. Chae, Remarks on the liouville type results for the compressible
navier-stokes equations in $\mathbb{R}^N$, {\it Nonlinearity,} {\bf
25} (2012) 1345--1349.






 \bibitem{Chamorro18}D. Chamorro. Espacios de Lebesgue y de Lorentz. Vol. 3.
hal-01801025v1 (2018).

\bibitem{Danchin04} R. Danchin, Local and global well-posedness results for flows
of inhomogeneous viscous fluids, Adv. Diff. Eq. 9 (2004) 353--386.

\bibitem{Danchin03}  R. Danchin, Density-dependent incompressible viscous fluids in
critical spaces, Proc. Roy. Soc. Edinburgh Sect. A, 133 (2003),
1311--1334

\bibitem{Galdi11} G.P. Galdi, An Introduction to the Mathematical Theory of the
Navier-Stokes Equations. Steady-State Problems, second edition,
Springer Monographs in Mathematics, Springer, New York, 2011,
xiv+1018 pp.

\bibitem{GS} Y. Giga and H. Sohr.
{ Abstract $L^p$ estimates for the Cauchy problem with applications
to the Navier-Stokes equations in exterior domains.} J. Funct. Anal.
102(1991), 72--94.

\bibitem{Kazhikov74} A.V. Kazhikov, Solvability of the initial-boundary value
problem for the equations of the motion of an inhomogeneous viscous
incompressible fluid, Dokl. Akad. Nauk SSSR 216 (1974) 1008-1010 (in
Russian).

\bibitem{Kim06} H. Kim, A blow-up criterion for the nonhomogeneous
incompressible Navier-Stokes equations, SIAM J. Math. Anal. 37
(2006) 1417-1434.

\bibitem{Jarrin20} O. Jarr\'{i}n,  A remark on the Liouville problem for stationary
Navier-Stokes equations in Lorentz and Morrey spaces. J. Math. Anal.
Appl.  486  (2020),  no. 1, 123871, 16 pp.

\bibitem{LS78}  O. Ladyzhenskaya, V. Solonnikov, Unique solvability of an initial
and boundary value problem for viscous incompressible
non-homogeneous fluids. J. Soviet Math. 9 (1978), 697--749.


\bibitem{LN21}  Z. Li, P. Niu, Notes on Liouville type theorems for the stationary compressible Navier-Stokes equations. Appl. Math. Lett.  114  (2021),
106908.

\bibitem{L2014}
D. Li, X. Yu, On some Liouville type theorems for the compressible
Navier-Stokes equations, {\it Discrete Contin Dyn Syst.,} {\bf 34}
(2014) 4719-4733.

\bibitem{LRM16} M. Loayza, M. A. Rojas-Medar. {A weak-$L^p$ Prodi-Serrin type regularity criterion for the
micropolar fluid equations.} J. Math. Phys. 57 (2016) 021512, 6 pp.


\bibitem{MXZ19} P.B. Mucha, L. Xue, X. Zheng, Between homogeneous
and inhomogeneous Navier-Stokes systems: the issue of stability. J.
Differential Equations  267  (2019),  no. 1, 307--363.

\bibitem{ON} R. O'Neil.
{Convolution operators and $L(p,q)$ spaces}. Duke Math. J. {\bf 30}
(1963) 129-142.


\bibitem{PY}
B. Pineau and  X. Yu,   A new Prodi-Serrin type regularity criterion
in velocity directions. J. Math. Fluid Mech., 20 (2018),
1737--1744.

\bibitem{Shi07} R. Shimada.
{On the $L_p-L_p$ maximal regularity for Stokes equations with Robin
boundary condition in a bounded domain.} Math. Methods Appl. Sci.
30(3) (2007), 257--289.

\bibitem{Simon90} J.Simon, Nonhomogeneous viscous incompressible fluids: Existence
of velocity, density, and pressure. SIAM J. Math. Anal. 21 (1990),
1093--1117.

\bibitem{SQ21} L. Sun, C. Qian, The global regularity for 3D inhomogeneous incompressible fluids
with vacuum, Appl. Math. Lett., 113 (2021) 106885.


\bibitem{Tr} H. Triebel.
{\it Theory of Function Spaces} Birkh\"{a}user Verlag, Basel-Boston,
1983.


\bibitem{YZ15} Z. Ye, X. Zhang, A note on blow-up criterion of strong solutions for
the 3D inhomogeneous incompressible Navier.Stokes equations with
vacuum, Math. Phys. Anal. Geom. 18 (2015) 24.





\end{thebibliography}
\end{document}